\documentclass{birkmult}
\usepackage{amssymb,amsmath,amsthm}

\newcommand{\E}{\mathbb{E}}
\newcommand{\N}{\mathbb{N}}
\newcommand{\R}{\mathbb{R}}
\newcommand{\C}{\mathbb{C}}

\renewcommand{\H}{\mathbb{H}}
\renewcommand{\Im}{\operatorname{Im}}
\newcommand{\G}{\mathfrak{G}}
\newcommand{\m}{\mathcal }

\newcommand{\x}{{\bf x}}

 \newtheorem{thm}{Theorem}[section]
 \newtheorem{cor}[thm]{Corollary}
 
 \newtheorem{prop}[thm]{Proposition}
 \theoremstyle{definition}
 \newtheorem{defn}[thm]{Definition}
 \theoremstyle{remark}
 \newtheorem{rem}[thm]{Remark}
 
 \numberwithin{equation}{section}

\begin{document}
%
%
%
%
%
%
%
%
%
\title[Stability on $\{0,1,2,\dotsc\}^S$]{Stability on $\{0,1,2,\dotsc\}^S$: birth-death chains \\and particle systems}
\author{Thomas M. Liggett}

\address{Department of Mathematics\\
University of California\\
Los Angeles CA 90095-1555}

\email{tml@math.ucla.edu}

\thanks{A. Vandenberg-Rodes	was partially supported by NSF grants DMS-0707226 and NSF-0949250.}
\author{Alexander Vandenberg-Rodes}
\address{Department of Mathematics\\
University of California\\
Los Angeles CA 90095-1555}
\email{avandenb@math.ucla.edu}


\date{2010}

\begin{abstract}
A strong negative dependence property for measures on $\{0,1\}^n$ -- \textit{stability} -- was recently developed in \cite{BBL}, by considering the zero set of the probability generating function. We extend this property to the more general setting of reaction-diffusion processes and collections of independent Markov chains. In one dimension the generalized stability property is now independently interesting, and we characterize the birth-death chains preserving it.
\end{abstract}

\maketitle

\section{Introduction}
In statistical physics a fundamental object of concern is the partition function, with its zeros having special relevance. For example, by introducing the effect of an external field, the partition function becomes a polynomial in the external field variable. As exemplified by the Lee-Yang circle theorem in the case of the Ising model \cite{leeyang}, the general location of partition function zeros can indicate possible phase transitions.

A related object in probability is the {\it probability generating function}. However, the locations of its zeros were little studied before the recent work of Borcea, Br\"and\'en, and Liggett. In \cite{BBL}, a strong negative dependence theory for measures on $\{0,1\}^n$ was obtained; in particular, it was shown that if the generating function (in $n$ variables) has no zeros with all imaginary parts positive, then the measure is negatively correlated in a variety of senses: negative association, ultra-log-concave rank sequence, Rayleigh property, and others.

The classification of linear transformations preserving the set of multivariate polynomials that are non-vanishing in circular regions was recently resolved in \cite{BBpolyaschur}, with the investigation providing a general account of such polynomials and unifying several Lee-Yang-type theorems \cite{BBpolyaschur2}.

Our results are as follows: Using this framework we will generalize the negative dependence result in \cite{BBL} to measures on $\{0,1,2,\dotsc\}^S$ -- $S$ countable -- with application to independent Markov chains and reaction-diffusion processes. The one-coordinate case is also independently interesting; more specifically, the probability measures under consideration can be decomposed into a sum of independent Bernoulli and Poisson random variables. 

Call such measures on $\{0,1,2,\dotsc\}$ {\it t-stable}. (The formal definition is given in Section 3.) In the last section we characterize most birth-and-death chains preserving this class of measures:
\begin{thm}\label{BD1}
The birth-death chain $\{X_t;t\geq 0\}$ with an infinite number of non-zero rates (e.g. irreducible on $\N$) preserves the class of t-stable measures if and only if the birth rates are constant and the death rates satisfy $\delta_k=d_1k + d_2k^2$ for some constants $d_1,d_2$.
\end{thm}
One example is the pure death chain with rates $\delta_k=k(k-1)/2$, which expresses the number of ancestral genealogies in Kingman's coalescent -- a well-studied model in mathematical biology \cite{kingman,griffiths,tavare}. In particular, by taking the initial number of particles to infinity, we obtain that the number of ancestors at any fixed time has the distribution of a sum of independent Bernoulli and Poisson random variables.

\section{Stability and Negative Association}
We first review the relationship -- established in \cite{BBL} -- between negative association and the zero set of generating functions for measures on $\{0,1\}^n$.
\begin{defn}
A polynomial $f(\x)\in \C[\x]=\C[x_1,\dotsc,x_n]$ is called {\it stable} if $f\neq 0$ on the set \[\H^n=\{(x_1,\dotsc,x_n)\in \C^n:\Im(x_j)>0 \ \forall j\}.\]
Let $\G[\x]$ be the set of all stable polynomials in the variable $\x$.

If $f$ has only real coefficients, it is also called {\it real} stable. The corresponding set of real stable polynomials is denoted $\G_\R[\x]$.
\end{defn}
Note that a univariate real stable polynomial can only have real zeros. 

One key fact from complex analysis is the (multivariate) Hurwitz's theorem on zeros of analytic functions: (see footnote 3 in \cite{choe})
\begin{thm}\label{hurwitz}
Let $\Omega$ be a connected open subset of $\C^n$. Suppose the analytic functions $\{f_k\}$ converge uniformly on compact subsets of $\Omega$ ({\it normal} convergence in the vocabulary of complex analysis). If each $f_k$ has no zeros in $\Omega$ then their limit $f$ is either identically zero, or has no zeros in $\Omega$. In particular, a normal limit of stable polynomials with bounded degree is either stable or $0$.
\end{thm}
For $\mu$ a probability measure on $\{0,1\}^n$, let \begin{equation}\label{generating}f_\mu(x_1,\dotsc,x_n)=\sum_{i_1,\dotsc, i_n=0}^1 \mu(i_1,\dotsc,i_n) x_1^{i_1}\cdots x_n^{i_n} = \E^\mu x_1^{\eta(1)}\cdots x_n^{\eta(n)}.\end{equation}
The last expression is just compact notation for the middle sum -- the $\eta(i)$ are the coordinate variables for $\mu$. $f_\mu$ is known as the {\it probability generating function} for $\mu$. With this identification between measures and polynomials, we will freely abuse notation by referring to measures with stable generating functions as {\it stable measures} (such measures are also termed {\it Strongly Rayleigh} \cite{BBL}, by their connection with the Rayleigh property).

The concept of stability easily generalizes to countably many coordinates -- a measure $\mu$ on $\{0,1\}^S$ is stable if every projection of $\mu$ onto finite subsets of coordinates is stable.

While the definition of stability is purely analytic, it implies two strong probabilistic conditions. Recall that a probability measure $\mu$ is negatively associated (NA) if, for all increasing continuous functions $F,G$ depending on disjoint sets of coordinates, \[\int FG d\mu\leq \int Fd\mu\int Gd\mu.\]
The following was proved in \cite{BBL}:
\begin{thm}\label{NA}
Suppose $f_\mu$ is stable. Then $\mu$ is NA.
\end{thm}

The second (and less difficult) probabilistic consequence of stability is the following \cite{liggett09,avr}.
\begin{thm}\label{sums}
Suppose $\mu$ is a measure on $\{0,1\}^S$ such that $f_\mu$ is stable. Then for any $T\subset S$, \[\sum_{i\in T}\eta(i)\stackrel{d}{=}\sum_{i\in T}\zeta_i,\] where the $\zeta_i$ are independent Bernoulli variables and the equality is in distribution.
\end{thm} 
For $S$ finite this latter result has been known since the work of L\'evy \cite{levy1}. See also Pitman \cite{pitman} for more combinatorial and probabilistic properties of stable generating functions, and the connection with P\'olya frequency sequences.

\section{Stable measures on $\{0,1,2,\dotsc\}^S$}

Suppose $\mu$ is a measure on $\{0,1,2,\dotsc\}^n$. The generating function of $\mu$ is now the formal power series
\begin{equation}\label{generating2}
f_\mu(x_1,\dotsc,x_n)=\sum_{i_1,\dotsc,i_n=0}^\infty\mu(i_1,\dotsc,i_n)x_1^{i_1}\cdots x_n^{i_n}.
\end{equation}
If $\mu$ has finite support, then $f_\mu$ is a polynomial. In this case, let $N$ be the maximum degree of $f_\mu$ in any of the variables $x_1,\dotsc,x_n$. We will want to represent $f_\mu$ by a multi-affine polynomial. To do this, we recall the {\it $k$-th elementary symmetric polynomial in $m$ variables} \begin{equation}\label{symm}
e_0=1,\quad e_{k}(x_1,\dotsc,x_m):=\sum_{1\leq i_1<i_2<\cdots<i_k\leq m}x_{i_1}x_{i_2}\cdots x_{i_k}.
\end{equation}
Then for a univariate polynomial $f(x)=\sum_{k=0}^N a_k x^k$ we define its {\it $N$-th polarization} as
\[Pol_N f(x_{1},\dotsc,x_{N}):=\sum_{k=0}^N \binom{N}{k}^{-1}a_k e_k(x_{1},\dotsc,x_{N}).\] The $N$-th polarization of a multivariate polynomial is then defined to be the composition of polarizations in each variable. 
By considering $x_1=\cdots=x_N=x$, notice that if $Pol_N f\in \G[\x]$ then $f\in \G[x]$. The converse also holds:
\begin{thm}
(Grace-Walsh-Szeg\"o). Suppose $f$ has degree at most N. Then $f$ is stable iff $Pol_N f$ is stable.
\end{thm}
Many proofs of this result and its equivalent forms exist; see the appendix of \cite{BBpolyaschur2}, or \cite[chapter 5]{complexpoly}.

\begin{defn}
We say that a function $f(\x)$ defined on $\C^n$ is {\it transcendental stable}, or {\it t-stable}, if there exist stable polynomials $\{f_m(\x)\}$ such that $f_m\rightarrow f$ uniformly on all compact subsets of $\C^n$ ($f$ can then be expressed as an absolutely convergent power series on $\C^n$). Let $\overline{\G[\x]}$ be the set of all t-stable functions -- this is also known as the Laguerre-P\'olya class \cite{levin}. Let $\overline{\G_\R[\x]}$ be the set of all {\it real} t-stable functions.
\end{defn}
We will again abuse notation and say that a measure $\mu$ on $\N^n$ is {\it transcendental stable} (or {\it t-stable}) if its generating function lies in $\overline{\G_\R[\x]}$. Similarly, if $\mu$ has finite support and its generating polynomial is stable then we say that $\mu$ is stable. Of course, a stable measure is automatically t-stable.

The papers by Borcea and Br\"and\'en \cite{BBpolyaschur,weyl} characterized the linear transformations preserving stable polynomials by establishing a bijection between linear transformations preserving n-variable stability and t-stable powers series in $2n$ variables. We will not require their full result here; however, the following characterization of t-stable powers series -- the technical cornerstone upon which the above bijection rests -- will be most useful.

Recall the standard partial order on $\N^n$: $\alpha\leq \beta$ if $\alpha_i\leq \beta_i$ for all $1\leq i\leq n$. Then for $\alpha,\beta\in \N^n$ and letting $\beta! = \beta_1! \cdots \beta_n!$, define 
\begin{equation}\label{notation}
(\beta)_\alpha = \frac{\beta!}{(\beta-\alpha)!} \mbox{ if }\alpha\leq\beta,\quad (\beta)_\alpha=0 \mbox{ otherwise.}
\end{equation}
\begin{thm}[Theorem 6.1 of \cite{BBpolyaschur}]\label{poly} Let $f(\x)=\sum_{\alpha\in \N^n}c_\alpha\x^\alpha$ be a formal power series in $\x$ with coefficients in $\R$. Set $\beta_m =(m,m,\dotsc,m)\in \N^n$. Then $f(\x)\in \overline{\G_\R[\x]}$ if and only if 
\begin{equation}\label{poly2}
f_m(\x):=\sum_{\alpha\leq \beta_m}(\beta_m)_\alpha c_\alpha\left(\frac{\x}{m}\right)^\alpha \in \G_\R[\x]\cup \{0\},
\end{equation}
for all $m\in\N$. In this case, the polynomials $f_m(\x)\rightarrow f(\x)$ uniformly on compact sets.
\end{thm}

This classification also has the following immediate consequence:
\begin{cor}\label{strongcor}

The class $\overline{\G_\R[\x]}$ is closed under convergence of coefficients. In particular, the set of t-stable probability measures on $\N^n$ is closed under weak convergence.
\end{cor}
\begin{proof}
Suppose that for each $n$, \[f^{(n)}(\x)=\sum_{\alpha\in \N^n}c^{(n)}_\alpha \x^\alpha \in \overline{\G_\R[\x]},\] with $c_\alpha^{(n)}\rightarrow c_\alpha$ for each $\alpha$. Then for each $m$ the stable polynomials $f^{(n)}_m(\x)$, defined in (\ref{poly2}) above, converge normally to the polynomial $f_m$ likewise obtained from $f$. Hurwitz's Theorem implies that each $f_m$ is stable, and applying Theorem \ref{poly} again we conclude that $f\in\overline{\G_\R[\x]}$.
\end{proof}
By the following proposition, we can say that a measure on $\{0,1,2,\dotsc\}^S$ is t-stable if every projection onto a finite subset of coordinates is a t-stable measure.
\begin{prop}\label{proj}
The class of t-stable measures is closed under projections onto subsets of coordinates.
\end{prop}
\begin{proof}
It suffices to consider projections of $n$ coordinates onto $n-1$ coordinates. Suppose that $\mu$ is a t-stable measure on $\N^n$. If $f(x_1,\dotsc,x_n)$ is its generating function, notice that the generating function of the projection of $\mu$ onto $\N^{n-1}$ is $f(x_1,\dotsc,x_{n-1},1)$. By Theorem \ref{poly}, it suffices to show that the approximating polynomials $f_m(x_1,\dotsc,x_{n-1},1)$ are stable. But this follows by considering the (complex) stable polynomials $f_m(x_1,\dotsc,x_{n-1},1+i/k)$ and applying Hurwitz's theorem as $k\rightarrow \infty$.
\end{proof}
We can now give an extension of Theorem \ref{NA}.
\begin{thm}\label{NA2}
Suppose $\mu$ is a t-stable probability measure on $\{0,1,2,\dotsc\}^S$. Then $\mu$ is NA.
\end{thm}
\begin{proof}
By the previous proposition and a limiting argument it is sufficient to show the result for measures on $\N^n$. Let $f(\x)$ be the generating function of $\mu$. By definition, $f\in\overline{\G_\R[\x]}$. Let $\{f_N(\x)\}$ be the stable polynomials converging to $f$ as in Theorem \ref{poly}, which we can assume are normalized so that $f_N(1)=1$. Let $\mu_N$ be the respective probability measures on $\{0,1,2,\dotsc,N\}^n$. Hence by the Grace-Walsh-Szeg\"o (GWS) theorem, $Pol_N{f_N}$ is the generating function for a stable measure $\tilde{\mu}_N$ on $\{0,1\}^{nN}$. Let \[\{\zeta_{ij};1\leq i\leq n,0\leq j\leq N\}\] be the coordinates of $\tilde{\mu}_N$, such that $\eta_i=\sum_j \zeta_{ij}$ is the $i$-th coordinate of $\mu_N$. Hence for bounded increasing functions $F$ and $G$ on $\{0,1,2,\dotsc\}^n$, depending on disjoint sets of coordinates, we have \begin{align*}
&\E^{\mu_N} [F(\eta_1,\dotsc,\eta_n)G(\eta_1,\dotsc, \eta_n)]\\
&=\E^{\mu_N}\bigg[F\bigg(\sum_j \zeta_{1j},\dotsc, \sum_j\zeta_{nj}\bigg)G\bigg(\sum_j \zeta_{1j},\dotsc, \sum_j\zeta_{nj}\bigg)\bigg]\\
&\leq \E^{\mu_N}\bigg[F\bigg(\sum_j \zeta_{1j},\dotsc, \sum_j\zeta_{nj}\bigg)\E^{\mu_N}\bigg[G\bigg(\sum_j \zeta_{1j},\dotsc, \sum_j\zeta_{nj}\bigg)\\
&=\E^{\mu_N} F(\eta_1,\dotsc,\eta_n)\E^{\mu_N} G(\eta_1,\dotsc, \eta_n)
\end{align*}
The inequality above follows because $F(x_{11}+\cdots+ x_{1N},\dotsc,x_{n1}+\cdots +x_{nN})$ is an increasing function in the $nN$ variables (similarly with $G$), and the $\zeta_{ij}$ are all negatively associated by Theorem \ref{NA}. The normal convergence of $f_N\rightarrow f$ implies the weak convergence $\mu_N\rightarrow \mu$, concluding the proof.
\end{proof}

We can also characterize all t-stable measures on one coordinate.

\begin{prop}\label{tstable}
A probability measure on $\{0,1,2,\dotsc\}$ is transcendental stable if and only if it has the same distribution as a (possibly infinite) sum of independent Bernoulli random variables and a Poisson random variable.
\end{prop}
\begin{proof}
Suppose $f$ is a t-stable generating function for a non-negative, integer valued random variable. By Theorem \ref{poly}, $f$ is a normal limit of univariate polynomials with all zeros on the negative real axis. An appeal to the classical theory of entire functions (e.g. \cite[VIII, Theorem 1]{levin}) indicates that $f$ can be expressed as the following infinite product: \[f(x) = Cx^q e^{\sigma x}\prod_{k=1}^\infty\left[1-\frac{x}{a_k}\right],\] for some $q\in \N$, $\sigma\geq 0$, $a_k<0$, and $\sum|a_k|^{-1}<\infty$. A little rearrangement -- using the fact that $f(1)=1$ -- gives the following alternative expression:\[f(x)=x^qe^{\sigma(x-1)}\prod_{k=1}^\infty[(1-p_k)+xp_k],\] where $p_k = 1/(1-a_k)$. This we recognize as the generating function for the sum of a non-negative constant $q$, independent Poisson($\sigma$) and Bernoulli($p_k$) random variables. Conversely, any generating function of this form with $\sum p_k<\infty$ is automatically t-stable, as $e^x$ is the normal limit of the polynomials $(1+x/n)^n$.
\end{proof}
By projecting onto finite subsets of coordinates (taking limits if need be) and setting all variables in the resulting generating function to be equal, we obtain the following extension of Theorem \ref{sums}:
\begin{cor}
Suppose $\mu$ is a t-stable measure on $\N^S$. Then for any $T\subset S$ the number of particles located in $T$ -- according to $\mu$ -- has the distribution of a sum of independent Bernoulli and Poisson random variables.
\end{cor}

\subsection{Markov processes and stability}
Suppose $\{\eta_t;t\geq 0\}$ is a Markov process on $\N^n$. 
We define the associated linear operator $T_t$ on power series with bounded coefficients by letting $T_t(\x^\alpha)$ be the generating function of $\{\eta_t|\eta_0=\alpha\}$ for each $\alpha\in\N^n$, and extending by linearity. This is well-defined because $\sum_{k\geq 0} P(\eta_t=k)=1$.
\begin{defn}[Preservation of stability]\label{preserve} We say that a Markov process $\{\eta_t;t\geq 0\}$ on $\N^n$ {\it preserves stability} if for any stable initial distribution, the distribution at any later time is t-stable. That is, the associated linear operator $T_t$ maps the set of stable polynomials with non-negative coefficients into the set of t-stable power series.
  
The process $\eta_t$ {\it preserves t-stability} if for any t-stable initial distribution, the distribution at a later time is again t-stable. That is, $T_t$ maps the set of t-stable power series with non-negative coefficients into itself.
\end{defn}
In fact, these two definitions are equivalent.
\begin{prop}\label{prop1}
A Markov process preserves t-stability if and only if it preserves stability.
\end{prop}
\begin{proof}
Only one direction needs proof. Assume the process preserves stability. Let $T_t$ be the associated linear operator, and $f=\sum_\alpha c_\alpha \x^\alpha$ be the generating function of a t-stable distribution; hence $c_\alpha\geq 0$ for all $\alpha$ and $\sum_\alpha c_\alpha = 1$. By Theorem \ref{poly} there are stable polynomials $f_n=\sum_{\alpha}c^{(n)}_\alpha \x^\alpha$ with $c_\alpha^{(n)}\rightarrow c_\alpha$, all $c_\alpha^{(n)}\geq 0$, and with $\sum_\alpha c_\alpha^{(n)}\leq 1$. Suppose that \[T_t(\x^\alpha) = \sum_\beta d_{\alpha,\beta}\x^\beta.\] Since probability is conserved, $\sum_\beta d_{\alpha,\beta}=1$, and hence by dominated convergence, \[\sum_{\alpha}c_\alpha^{(n)}d_{\alpha,\beta}\longrightarrow\sum_\alpha c_\alpha d_{\alpha,\beta}\quad \mbox{as }n\rightarrow\infty.\] In other words, the coefficients of $T_tf_n$ (which is t-stable by assumption), converge to the coefficients of $T_t f$. $T_t f$ is then t-stable by Corollary \ref{strongcor}.
\end{proof}
We now give a couple examples.
\subsection{Independent Markov chains}
\noindent Suppose $\{X_t(1),X_t(2),\dotsc\}$ is a collection of independent Markov chains on $S$ with identical jump rates. Set \[\eta_t(x) = \sum_{i\geq 1} 1_{\{X_t(i)=x\}},\] so that the resulting process is a collection of particles on $S$ jumping independently with the same rates. This is well defined as long as $\eta_t(x)<\infty$ for all $x\in S,\ t\geq 0$ -- one possibility is to restrict initial configurations to the space $E_0$ defined below for reaction-diffusion processes.
\begin{prop}\label{indep}
The process $\{\eta_t;t\geq0\}$ preserves t-stability. Hence, assuming that the initial distribution is t-stable, the distribution at any time is negatively associated by Theorem \ref{NA2}.
\end{prop}
\begin{proof}
Let $\mu_t$ be the distribution of $\eta_t$, with $\mu_0$ t-stable. We need to show that for each finite $T\subset S$, the projection $\mu_t\vert_T$ is t-stable. Taking finite $T\subset S_1\subset S_2\subset \cdots$ with each $S_n$ finite and $S_n\nearrow S$, we can approximate $\mu_t\vert_T$ by the sequence $\mu^{(n)}_t\vert_T$, with each $\mu^{(n)}_t$ the distribution of the independent Markov chain process on $S_n$ given initial distribution $\mu_0\vert_{S_n}$ and jumps restricted to staying inside $S_n$. Hence by Corollary \ref{strongcor} we can assume finite $S$.

Suppose now that $S=[n]$, and only jumps from site $1$ to site $2$ are allowed. In this case, assuming a jump rate $q(1,2)$, each particle at $x$ independently has probability $p:=1-e^{-tq(1,2)}$ of moving to $y$. Hence the associated linear operator $T_t$ takes \[x_1^{\alpha_1}x_2^{\alpha_2}\cdots x_n^{\alpha_n}\mapsto (px_2+(1-p)x_1)^{\alpha_1}x_2^{\alpha_2}\cdots x_n^{\alpha_n},\] that is,
\[T_tf(x_1,\dotsc,x_n)=f(px_2+(1-p)x_1,x_2,\dotsc,x_n),\] which preserves the class of stable polynomials.
By permuting variables, this argument also holds for general $i,j\in [n]$.

Recalling the Banach space $C_0(\N^n)$ of functions that vanish at infinity, consider the strongly continuous contraction semi-groups $S^{i,j}(t)$ on $C_0(\N^n)$ defined by \[S^{i,j}(t)f(\eta) = \E^\eta f(\eta^{i,j}_t),\] where $\{\eta^{i,j}_t;t\geq0\}$ is the (Feller) process of independent Markov chains which only allow jumps from site $i$ to site $j$. Then for a t-stable intial distribution $\mu$, we just showed that $\mu S^{i,j}(t)$ -- the distribution of $\eta^{i,j}_t$ assuming initial distribution $\mu$ -- is again t-stable. By Trotter's product theorem \cite[p. 33]{ethierkurtz}, the process allowing the jumps $\{i\mapsto j\},\{k\mapsto l\}$ has semigroup 
\[S(t)= \lim_{n\rightarrow\infty}\Big[S^{i,j}\Big(\frac tn\Big)S^{k,l}\Big(\frac tn\Big)\Big]^n.\] 
Hence by Corollary \ref{strongcor}, $\mu S(t)$ is again t-stable. Including all the possible jumps one-by-one, we conclude that the whole process $\{\eta_t;t\geq 0\}$ on $\N^n$ preserves t-stability. 
\end{proof}

\begin{rem}Let $\m M$ be the set of probability measures on $\N^S$ described by random configurations $\eta$ with coordinates \[\eta(k)=\sum_{i} 1_{\{Y_i=k\}},\] where $Y_1,Y_2,\dotsc$ are independent random variables with values in $S\cup\{\infty\}$. In \cite{liggettNA} it was shown that $\m M$ is preserved by the process of independent Markov chains, and that measures in $\m M$ are NA. Proposition \ref{indep} is a generalization of this result, since it is easily checked that the class $\m M$ is contained in the class of t-stable measures. Indeed, if $\mu\in \m M$ and $S=[n]$, then $\mu$ has a generating function of the form 
\begin{align*}
&\E x_1^{\sum_i 1_{\{Y_i=1\}}}\cdots x_n^{\sum_i 1_{\{Y_i=n\}}}\\
&=\prod_i[P(Y_i=\infty)+P(Y_i=1)x_1+\cdots+P(Y_i=n)x_n],
\end{align*}
by the independence of the $Y_i's$. Furthermore, (non-constant) product measures for which each coordinate is a sum of independent Bernoulli and Poisson measures are t-stable, but are not contained in the class $\m M$.
\end{rem}

\subsection{Reaction-diffusion processes}
In addition to having the motion of particles following independent Markov chains, we can also allow particles to undergo a reaction at each site. Let $p(i,j)$ be transition probabilities for a Markov chain on $S$. Given a state $\eta\in \N^S$, we consider the following evolution:
\begin{enumerate}
\item at rate $\beta^i_{\eta(i)}$ a particle is created at site $i$,
\item at rate $\delta^i_{\eta(i)}$ a particle at site $i$ dies.
\item at rate $\eta(i)p(i,j)$ a particle at site $i$ jumps to site $j$.
\end{enumerate}
The most common example is the {\it polynomial model} of order $m$, where the birth-death rates for each site are \[\beta_k=\sum_{j=0}^{m-1}b_j k(k-1)\cdots (k-j+1),\quad \delta_k=\sum_{j=1}^m d_j k(k-1)\dotsc(k-j+1).\]
Reaction-diffusion processes originated as a model for chemical reactions \cite{schlogl}, and subsequent work by probabilists has focused on the ergodic properties \cite{ding, chen, athreya}. 

To construct the process, we require a strictly positive sequence $k_i$ on the index set $S$, and a positive constant $M$ such that \[\sum_j p(i,j)k_j\leq M k_i,\quad i\in S.\] 
Furthermore, the birth rates must satisfy \[\sum_i \beta^i_0 k_i<\infty.\]
Then we take the state space of the process to be \[E_0=\{\eta\in \N^S: \sum_i \eta(i)k_i<\infty\}.\] See \cite[chapter 13.2]{chen} for the details of the construction.

In a very simple case we have preservation of t-stability:
\begin{prop}\label{reaction}
Suppose the reaction-diffusion process on $\N^S$ is well-constructed with $\beta^i_k=b^i$ and $\delta^i_k=d^i k$ -- this is the polynomial model of order $1$ and site-varying reaction rates. Then the process preserves t-stability, and hence -- assuming a t-stable initial configuration -- its distribution at any time is negatively associated.
\end{prop}

\begin{proof}
The strategy here is the same as with Proposition \ref{indep}. To reduce to a reaction-diffusion process on a finite number of sites, we approximate using the construction in \cite[Theorem 13.8]{chen}. Furthermore, on the locally compact space $\N^n$, the reaction-diffusion process with at most constant birth and linear death rates is now a Feller process, as can be seen from \cite[Theorem 3.1, Ch. 8]{ethierkurtz}. Hence by Trotter's product formula and Corollary \ref{strongcor} we only need to show that the following processes preserve stability on $\N^n$:
\begin{enumerate}
\item constant birth rate $b^i$ at a single site $i$.
\item linear death rates at a single site $i$ ($\delta^i_k=d^i k$),
\item jumps from site $i$ to site $j$ at rate $\eta(i)p(i,j)$
\end{enumerate}

For (1), we note that with a constant birth rate, at time $t$ a Poisson$(b^i t)$ number of particles has been added to the system -- i.e. the original generating function is multiplied by $e^{b^i t(x_i-1)}$, preserving t-stability.

(2) can be thought of as the process in which each particle at site $i$ dies independently at rate $d^i$. Hence the associated linear transform is defined by \[T_t(x_1^{\alpha_1}\cdots x_i^{\alpha_i}\cdots x_n^{\alpha_n}) = x_1^{\alpha_1}\cdots[1-e^{-d^i t}+e^{-d^i t}x_i]^{\alpha_i}\cdots x_n^{\alpha_n}.\] As the affine transformation $x\mapsto ax+(1-a)$, ($a>0$) maps the upper half plane onto itself, $T_t$ preserves preserves the class of stable polynomials.

Finally, we note that (3) was already seen to preserve stability from the proof of Proposition \ref{indep} for independent Markov chains.
\end{proof}
Other reaction-diffusion processes do not preserve stability in general. On one coordinate, a reaction-diffusion process is just a birth-death chain, so by Theorem \ref{BD1} the only possible generalization would be to quadratic death rates. In this case, unfortunately, the associated linear transformation $T_t$ will {\it not} preserve stable polynomials with positive roots; indeed, assuming death rates $\delta_k=k(k-1)$, quadratic polynomials with a double root inside the interval $(0,1)$ will not be mapped to stable polynomials under $T_t$. A much more complicated example -- which we shall not reproduce here -- shows that quadratic death rates on multiple sites does not preserve the class of stable probability measures.

\section{Birth-Death Chains}
Our goal in this section is to prove Theorem \ref{BD1}. As just noted above, in the case of quadratic death rates the associated linear transformation does not preserve all polynomials with real zeros, and hence we cannot rely on the classification theory in \cite{BBpolyaschur}. From the probabilistic point of view, it would be useful to have a similar theory for linear transformations on polynomials with positive coefficients; however, in what follows we will make do with several perturbation arguments -- the main idea being that a polynomial's roots move continuously under changes to its coefficients.

Recall that a (continuous-time) birth-death chain $\{X_t;t\geq 0\}$ is a Markov process on the non-negative integers with transitions \[k\mapsto k+1 \quad\mbox{at rate }\beta_k,\quad k\mapsto k-1 \quad\mbox{at rate }\delta_k, \quad \mbox{with }\delta_0=0.\]

We only consider rates $\{\beta_k,\delta_k\}$ such that the process does not blow up in finite time -- see Chapter 2 of \cite{Li} for the necessary and sufficient conditions, as well as for the construction of the process. 

Because of our definition of ``preserving stability'' (Definition \ref{preserve}), Theorem \ref{BD1} and the auxiliary results below only make sense for birth-death chains with infinitely many non-zero rates (e.g. when the chain is irreducible on $\N$). Indeed, if one fixes $n$ and considers rates $\beta_k=n-k$ and $\delta_k=k$ for $0\leq k\leq n$, with all other rates zero, then it can be seen that this process preserves all stable polynomials of degree at most $n$.

The generating function at time $t$ is given by
$$\phi(t,z)=\sum_{k=0}^\infty P(X_t=k)z^k.$$ 

By Theorem 2.14 of \cite{Li}, the transition probabilities \[p_t(j,k)=P(X_t=k|X_0=j)\] are continuously differentiable in $t$
and satisfy the Kolmogorov backward equations
$$\frac d{dt}p_t(j,k)=\beta_jp_t(j+1,k)+\delta_jp_t(j-1,k)-(\beta_j+\delta_j)p_t(j,k).$$
By Theorem 2.13 of \cite{Li}, $|p_t(j,k)-p_s(j,k)|\leq 1-p_{|t-s|}(j,j).$ It follows that
$$\frac{\partial}{\partial t}\phi(t,z)=\sum_{k=0}^\infty\frac d{dt}P(X_t=k)z^k$$
for $|z|<1$, provided that $X(0)$ is bounded. Iterating this argument, one sees that $\phi(t,x)$ is $C^2$ on $[0,\infty)\times
\{z:|z|<1\}$ if $X(0)$ is bounded.

\begin{prop}\label{onlyif} Suppose the birth-death process $\{X_t;t\geq 0\}$ preserves stability. Then there exist constants $b_0,b_1,b_2,d_1,d_2$ so that
\[\beta_k=b_0+b_1k+b_2k^2\quad\text{and}\quad\delta_k=d_1k+d_2k^2.\]
\end{prop}
\begin{proof} Suppose
$$\phi(0,z)=c\prod_{k=1}^n(z-z_k),$$
where $c>0$, $n\geq 3$, and $-1<z_1,...,z_n<0.$ By assumption, the generating function $\phi(t,z)$ of $X_t$ has only real roots for $t>0$. If $z_1=z_2=w$ is a root of $\phi(0,z)$
of multiplicity exactly two, and $\epsilon$ is small enough that $|z_k-w|>\epsilon$ for $k\geq 3$, then Rouch\'e's Theorem implies that for sufficiently small $t>0$, $\phi(t,z)$ has exactly two roots in the disk $\{z:|z-w|<\epsilon\}$. Therefore, for small $t>0$, there exist real $z(t)$ so that $\phi(t,z(t))=0$ and $\lim_{t\downarrow 0}z(t)=w$. By Taylor's Theorem, there exist $s(t)\in [0,t]$ and $y(t)$ between $z(t)$ and $w$ so that
$$\begin{aligned}0=&\phi(t,z(t))=t\frac{\partial\phi}{\partial t}(0,w)+\frac 12t^2\frac{\partial^2\phi}{\partial t^2}(s(t),y(t))\\&+t(z(t)-w)\frac{\partial^2\phi}{\partial t\partial z}(s(t),y(t))+\frac 12(z(t)-w)^2\frac{\partial^2\phi}{\partial z^2}(s(t),y(t)).\end{aligned}$$
Dividing by $t$ and letting $t\downarrow 0$ leads to
$$2\frac{\partial\phi}{\partial t}(0,w)\bigg/\frac{\partial^2\phi}{\partial z^2}(0,w)=-\lim_{t\downarrow 0}\frac{(z(t)-w)^2}t\leq0.$$
Noting that
$$\frac{\partial^2\phi}{\partial z^2}(0,w)=2c\prod_{k=3}^n(z-z_k),$$
we see that
\begin{equation}\label{tpartial}\frac{\partial\phi}{\partial t}(0,w)\end{equation}
changes sign when $z_3$ crosses $w$, and hence is zero when $z_3=w$. 

To exploit this fact, we need to compute (\ref{tpartial}). Recall the $k^{th}$ elementary symmetric polynomials $e_k(x_1,\dotsc,x_n)$ defined in (\ref{symm}).
If $\mu$ is the distribution of $X_0$ and $X_0\leq n$, then
$$\phi(0,z)=\sum_{k=0}^n\mu(k)z^k=c\prod_{k=1}^n(z-z_k)=c\sum_{k=0}^n(-1)^k e_k(z_1,...,z_n)z^{n-k},$$
so
$$\mu(k)=c(-1)^{n-k}e_{n-k}(z_1,...,z_n).$$
Therefore for $|z|<1$, by the Backward equations,
\begin{align*}\frac{\partial\phi}{\partial t}(0,z)=&\sum_{l=0}^\infty[\mu(l-1)\beta_{l-1}+\mu(l+1)\delta_{l+1}-\mu(l)(\beta_l+\delta_l)]z^l\\
=&c(1-z)\sum_{k=0}^n(-1)^{n-k}e_{n-k}(z_1,...,z_n)[\delta_kz^{k-1}-\beta_kz^k].\end{align*}
It follows that the expression on the right is zero if $z=z_1=z_2=z_3=w$ for any values of $w,z_4,...,z_n\in (-1,0)$. In this case,
$$e_k(z_1,...,z_n)=\sum_iw^i\binom 3ie_{k-i}(z_4,...,z_n),$$
where $i$ ranges from $\max(0,k+3-n)$ to $\min(k,3)$, so
$$\sum_{k=0}^n\sum_i(-1)^{n-k}\binom 3ie_{n-k-i}(z_4,...,z_n)[\delta_kw^{k+i-1}-\beta_kw^{k+i}]\equiv 0.$$ Interchanging the order of summation and letting $k\mapsto k-i$, we see that the coefficient of each of the $e_{n-k}$'s is zero:
$$\sum_i\binom 3i(-1)^i[\delta_{k-i}-\beta_{k-i}w]=0,$$
or equivalently $\delta_k-3\delta_{k+1}+3\delta_{k+2}-\delta_{k+3}=0$ and $\beta_k-3\beta_{k+1}+3\beta_{k+2}-\beta_{k+3}=0$. So the birth rates $\beta_k$ and death rates $\delta_k$ are quadratic functions of $k$, and the boundary term $\delta_0=0$ enforces $\delta_k=d_1k + d_2k^2$ for some $d_1,d_2$.
\end{proof}
With the next proposition we resolve the ``only if'' part of Theorem \ref{BD1}.
\begin{prop} The birth-death chain preserves stability only if the birth rate is constant.
\end{prop}
\begin{proof}
Assuming that the chain preserves stability, we will show the birth rates $\beta_k$ satisfy $\beta_{k}\geq \beta_{k+1}$ for each $k$, so by Proposition \ref{onlyif} $\beta_k$ is constant.

By iterating the Kolmogorov backward equations, one can obtain the following approximations for small $t>0$:\[p_t(k,k+1) = t(\beta_k+o(1)),\quad p_t(k,k+2)=\frac{t^2}2 (\beta_k\beta_{k+1}+o(1)).\] Similarly, \[p_t(k,k-j)=O(1)t^j,\] where $O(1)$ denotes a uniformly bounded quantity, and $o(1)\rightarrow 0$, as $t\rightarrow 0$. 

Suppose that we start the chain with $k$ particles; the initial distribution has generating function $f(x)=x^k$. We also can assume that $\beta_k,\beta_{k+1}>0$. Then the generating function for small $t>0$ will be:
\[f_t(x)=\cdots+(1+o(1))x^k+(\beta_k+o(1))tx^{k+1}+(\beta_k\beta_{k+1}+o(1))\frac{t^2}{2}x^{k+2}+\cdots\]
Since $f_t(x)$ is t-stable, by Theorem \ref{poly} the following polynomial has all real, negative roots: \[f_{t,k+2}(x)=\cdots+(1+o(1))x^k+\frac{2(\beta_k+o(1))}{k+2}tx^{k+1}+\frac{\beta_k\beta_{k+1}+o(1)}{(k+2)^2}t^2x^{k+2}.\]
As the hidden coefficients are $o(1)$, Rouch\'e's Theorem implies that $k$ roots of $f_{t,k+2}$ are also $o(1)$. Thus the remaining two roots $a,b$ satisfy 
\begin{align*}
a+b&=\frac{2(k+2)+o(1)}{\beta_{k+1}t} \\
ab&=\frac{(k+2)^2+o(1)}{\beta_k\beta_{k+1}t^2}.
\end{align*}
Solving for real $a,b$ implies that the discriminant \[4t^{-2}(k+2)^2[\beta_{k+1}^{-2}-(\beta_k\beta_{k+1})^{-1}+o(1)]\geq 0,\quad \mbox{for small }t>0.\] Taking $t\rightarrow 0$, we conclude that $\beta_k\geq \beta_{k+1}$.
\end{proof}

We now concentrate on the ``if'' part of Theorem \ref{BD1}. Here is a useful fact about quadratic death rates:
\begin{prop}\label{wright}
The generating function of the birth-death chain with rates $\beta_k=0$, $\delta_k=k(k-1)$ satisfies the Wright-Fisher partial differential equation:
\begin{equation*}
\frac{\partial}{\partial t}\phi(t,z)=z(1-z)\frac{\partial^2}{\partial z^2} \phi(t,z).
\end{equation*}
\end{prop}
\begin{proof}
Let $p_t(l,k)=P(X_t=k|X_0=l)$ be the transition probabilities. Since the process stays bounded we can use the Kolmogorov Forward equation	:

\begin{equation}
\frac d{dt}p_t(l,k)=(k+1)kp_t(l,k+1)-k(k-1)p_t(l,k)
\end{equation}

Then, since $\mu_t(k)=\sum_l\mu(l)p_t(l,k)$, we have
\begin{align*}
\frac{\partial}{\partial t}\phi(t,z) = &\sum_{k,l}\mu(l)\frac d{dt}p_t(l,k)z^k \\
=&\sum_{k,l}\mu(l)[(k+1)kp_t(l,k+1)-k(k-1)p_t(l,k)]z^k\\
=&z(1-z)\sum_{k,l}\mu(l)k(k-1)p_t(l,k)z^{k-2}=z(1-z)\frac{\partial^2}{\partial z^2}\phi(t,z).
\end{align*}
\end{proof}

\begin{prop}\label{quaddeath}
The birth-death chain with rates $\beta_k=0$, $\delta_k=k(k-1)$ preserves stability.
\end{prop}
\begin{proof}
As above, let $\phi(t,z)$ be the generating function of the chain at time $t$, and now assume that $\phi(0,z)$ has only real roots. Setting \[\tau =\inf\{t\geq 0;\ \phi(t,z) \mbox{ is not stable}\},\] by Hurwitz's Theorem (Theorem \ref{hurwitz}) $\phi(\tau,z)$ is stable. Hence it suffices to prove that for any stable initial distribution there exists an $\epsilon>0$ such that $\phi(t,z)$ is stable for all $0<t<\epsilon$.

For any root $w<0$ of $\phi(0,t)$, we can write 
\begin{equation}\label{phi}
\phi(0,z)=(z-w)^nq(z),
\end{equation}
where $n$ is the multiplicity and $q(w)\neq 0$. We will show that for all small enough $t$ the generating function $\phi(t,z)$ has n distinct real roots of distance approximately $\sqrt{t}$ from $w$. This follows directly from showing that 
\begin{equation}\label{nicepoly}
\phi(t,w+\alpha t^{1/2})=t^{n/2}p(\alpha)+o(t^{n/2}),\mbox{ as }t\downarrow 0,
\end{equation}
where $p(\alpha)$ has only simple real roots. Indeed, (\ref{nicepoly}) immediately implies that for all small enough $t$, $\phi(t,z)$ changes sign $n$ times near $w$. Since we can do this for each $w$, this shows that $\phi(t,z)$ has only real roots.

Set $\kappa=\lfloor n/2\rfloor$. To show (\ref{nicepoly}), we Taylor expand $\phi(t,z)$ in the first variable, i.e., for small $t>0$ \[\phi(t,z)=\sum_{k=0}^\kappa \frac{\partial^k\phi}{\partial t^k}(0,z)\frac{t^k}{k!}+o(t^{n/2}).\]
Now $k$ iterations of Proposition \ref{wright} shows that 
\begin{equation}\label{expand}
\frac{\partial^k\phi}{\partial t^k}(0,z)=[z(1-z)]^k\frac{\partial^{2k}\phi}{\partial z^{2k}}(0,z)+\sum_{j=1}^{2k}f_{j,k}(z) \frac{\partial^{2k-j}\phi}{\partial z^{2k-j}}(0,z),
\end{equation}
where the $f_{j,k}(z)$ are polynomials in $z$. 
With $z=w+\alpha t^{1/2}$ we have, by the product rule,
\begin{equation}\label{prodrule}
\frac{\partial^{2k-j}\phi}{\partial z^{2k-j}}(0,w+\alpha t^{1/2})=\frac{n!}{(n-2k+j)!}(\alpha t^{1/2})^{n-2k+j}q(w)+o(t^{\frac{n-2k+j}{2}}).
\end{equation}
We've also used the fact that $q(w+\alpha t^{1/2})=q(w)+o(1)$, because $q$ has no zeros at $w$.
From (\ref{prodrule}) we see that \[f_{j,k}(w+\alpha t^{1/2}) \frac{\partial^{2k-j}\phi}{\partial z^{2k-j}}(0,w+\alpha t^{1/2})=o(t^{\frac{n-2k}{2}}),\]and can be ignored for $j>0$.
Putting this all together, we have 
\[\sum_{k=0}^\kappa \frac{t^k}{k!}\frac{\partial^k\phi}{\partial t^k}(0,w+\alpha t^{1/2})= t^{n/2}\sum_{k=0}^\kappa \frac{n!}{k!(n-2k)!}(-1)^k[w(w-1)]^k\alpha^{n-2k}q(w) + o(t^{n/2}).\] Recalling the (physics) Hermite polynomial of order $n$,
\[H_n(\alpha)=\sum_{k=0}^\kappa (-1)^k\frac{\alpha^{n-2k}n!}{k!(n-2k)!},\] we obtain 
\[\phi(t,w+\alpha t^{1/2})=t^{n/2}[w(w-1)]^{n/2} H_n(\alpha/2\sqrt{w(w-1)})q(w)+o(t^{n/2}).\] $H_n$ is well known (e.g. \cite[\S 3.3]{szego}) to have $n$ distinct real roots, which, in the case that $w<0$ shows (\ref{nicepoly}). 

The case $w=0$ requires an additional argument. It will turn out, assuming $\phi(0,z)=z^nq(z)$, that $\phi(t,z)$ has $n-1$ distinct negative zeros with distance of order $t$ from zero, not $t^{1/2}$.

To show this we again Taylor expand in $t$. By Proposition \ref{wright} we see that 
\begin{equation*}
\frac{\partial^k\phi}{\partial t^k}(0,z)=z(1-z)\frac{\partial^2}{\partial z^2}\left[\cdots z(1-z)\frac{\partial^2}{\partial z^2} z^nq(z)\right],
\end{equation*}
where the operator $z(1-z)\partial_z^2$ is applied $k$ times to $z^nq(z)$. If we evaluate this at $z=\alpha t$, we obtain \[\frac{\partial^k\phi}{\partial t^k}(0,\alpha t)=[n(n-1)]\cdots [(n-k+1)(n-k)](\alpha t)^{n-k}\phi(0)+o(t^{n-k}).\] Thus the Taylor expansion gives \[\phi(t,\alpha t)=t^n\sum_{k=0}^{n-1}\frac{[n]_k[n-1]_k}{k!}\alpha^{n-k} + o(t^n),\]
where \[[n]_k=n(n-1)\cdots (n-k+1)\] is the falling factorial. The sum can be rewritten as \[\alpha^nn! _1\!F_1[1-n,1,-1/\alpha],\]where $_1\!F_1[a,b,x]$ is Kummer's hypergeometric function of the first kind. Now $_1\!F_1[1-n,1,-1/x]$ has $n-1$ distinct negative zeros $\{z_1,\dotsc z_{n-1}\}$ \cite[pp. 103]{slater}. Hence for $t$ small enough we can interlace $n-1$ zeros of $\phi(t,z)$ between the points $\{0,z_1,\dotsc,z_{n-1}\}$, and so the single remaining zero of $\phi(t,z)$ near 0 must also be real.
\end{proof}

\begin{proof}[Proof of ``if" direction in Theorem \ref{BD1}] By Proposition \ref{reaction}, the birth-death chain with constant birth and linear death rates preserves t-stability, and we just showed that the pure quadratic death chain preserves stability (and hence t-stability). However, the latter chain is no longer a Feller process, so we cannot immediately apply Trotter's product formula -- as we did with reaction-diffusion processes and independent Markov chains -- to combine the two processes. Indeed, it is well known that a pure quadratic death chain comes down from infinity in finite time, in the sense that $\liminf_{k\rightarrow\infty} p_t(k,1)>0$ for each $t>0$ \cite{kingman}. 

We rectify this situation by considering the Banach space $l^1(\N)$ of absolutely summable sequences. Let $X^{(1)}_t,\ X^{(2)}_t$, and $X^{(3)}_t$ be the birth-death chains with respective rates $\{\beta^{(1)}_k=b_0,\ \delta^{(1)}_k=d_1 k\}$, $\{\beta^{(2)}_k=0,\ \delta^{(2)}_k=d_2k(k-1)\}$, and $\{\beta^{(3)}_k=b_0,\ \delta^{(3)}_k=d_1 k+d_2k(k-1)\}$. With \[P^{(i)}(t)f(x)=\sum_y f(y)P(X^{(i)}_t=x|X^{(i)}_0=y)\] as the (adjoint) strongly continuous contraction semigroups on $l^1(\N)$, we consider the infinitesimal generators as the $l^1$ limit \[\Omega^{(i)}f=\lim_{t\downarrow 0}\frac{P^{(i)}(t)f-f}{t}.\] See \cite{reuter} for the theory of adjoint semigroups of Markov chains.

Let 
\begin{align*}D_0&=\{f\in l^1(\N);\ f(x)=0 \mbox{ for all but finitely many }x\},\\
D_e&=\{f\in l^1(\N);\ |f(x)|\leq Ce^{-x}\}, \mbox{ $C$ depending only on $f$, and}\\
\mathcal{D}(\Omega^{(i)}) &=\{f\in l^1(\N);\ \lim_{t\downarrow 0}t^{-1}(P^{(i)}(t)f-f) \mbox{ exists as an $l^1$ limit.}\}
\end{align*}
By explicit calculation, it can be seen that $D_0\subset D_e\subset \mathcal{D}(\Omega^{(i)})$ for each $i$, $P^{(i)}(t):D_0\rightarrow D_e$, and for $f\in D_e$, 
\[\Omega^{(i)}f(x)=\delta^{(i)}_{x+1}f(x+1)+\beta^{(i)}_{x-1}f(x-1)-[\beta^{(i)}_x+\delta^{(i)}_x]f(x).\]
By \cite[Prop. 3.3 of Ch. 1]{ethierkurtz}, $D_e$ is a core for all three generators. Also, \[\Omega^{(1)}+\Omega^{(2)}=\Omega^{(3)} \quad\mbox{ on }D_e,\] so we can apply Trotter's product formula to conclude preservation of t-stability for $X^{(3)}_t$.
\end{proof}



\begin{thebibliography}{1}
\bibitem{athreya} S. R. Athreya, J. M. Swart, \textit{Branching-coalescing particle systems.} Probab. Th. Rel. Fields \textbf{131} (2005), 376--414.
\bibitem{BBpolyaschur} J. Borcea, P. Br\"and\'en, \textit{The Lee-Yang and P\'olya-Schur programs. I. Linear operators preserving stability.} Invent. Math. \textbf{177} (2009), 541--569.
\bibitem{BBpolyaschur2} J. Borcea, P. Br\"and\'en, \textit{The Lee-Yang and P\'olya-Schur programs. II. Theory of stable polynomials and applications.} Comm. Pure Appl. Math. \textbf{62} (2009), 1595--1631.
\bibitem{weyl} J. Borcea, P. Br\"and\'en, \textit{Multivariate P\'olya-Schur classification problems in the Weyl algebra.} Proc. London Math. Soc. \textbf{101} (2010), 73--104.
\bibitem{BBL} J. Borcea, P. Br\"and\'en, T. M. Liggett, \textit{Negative dependence and the geometry of polynomials}, J. Amer. Math. Soc. \textbf{22} (2009), 521--567.
\bibitem{chen} M. F. Chen, \textit{From Markov Chains to Non-Equilibrium Particle Systems.} World Scientific, 1992.
\bibitem{choe} Y.-B. Choe, J. G. Oxley, A. D. Sokal, D. G. Wagner, \textit{Homogeneous multivariate polynomials with the half-plane property.} Adv. in Appl. Math, \textbf{32} (2004), 88--187.
\bibitem{ding} W. Ding, R. Durrett, T. M. Liggett, \textit{Ergodicity of reversible reaction diffusion
processes.} Probab. Th. Rel. Fields \textbf{85} (1990), 13–-26.
\bibitem{ethierkurtz} S. N. Ethier, T. G. Kurtz, \textit{Markov Processes, Characterization and Convergence.} Wiley Series in Probability and Mathematical Statistics, 1986.
\bibitem{griffiths} R. C. Griffiths, \textit{Lines of descent in the diffusion approximation of neutral Wright-Fisher models.} Theor. Pop. Biol. \textbf{17} (1980), 37--50.
\bibitem{kingman} J. F. C. Kingman, \textit{The coalescent.} Stoch. Proc. Appl. \textbf{13} (1982), 235-248.
\bibitem{leeyang} T. D. Lee, C. N. Yang, \textit{Statistical theory of equations of state and phase transitions. II. Lattice gas and Ising model.} Phys. Rev. \textbf{87} (1952), 410--419.
\bibitem{levin} B. Ja. Levin, \textit{Distribution of Zeros of Entire Functions.} American Mathematical Society, 1980.
\bibitem{levy1} P. L\'evy, \textit{Sur une propri\'et\'e de la loi de Poisson relative aux petites probabilit\'es.} Soc. Math. de France, C. R. sc\'eances de l'annee 1936 (1937), p. 29.
\bibitem{liggettNA} T. M. Liggett, \textit{Negative correlations and particle systems.} Markov Proc. Rel. Fields \textbf{8} (2002), 547--564.
\bibitem{liggett09} T. M. Liggett, \textit{Distributional limits for the symmetric exclusion process.} Stoch. Proc. Appl. \textbf{118} (2008), 319--332. 
\bibitem{Li} T. M. Liggett, \textit{Continuous Time Markov Processes: An Introduction.} AMS Graduate Studies in Mathematics, volume 113, 2010.
\bibitem{pitman} J. Pitman, \textit{Probabilistic bounds on the coefficients of polynomials with only real zeros.} J. Comb Theory A. \textbf{77} (1997), 279--303.
\bibitem{reuter} G. E. H. Reuter, \textit{Denumerable Markov processes and the associated contraction semigroups on $l$.} Acta Math. \textbf{97} (1957), 1--46.
\bibitem{schlogl} F. Schl\"ogl, \textit{Chemical reaction models for non-equilibrium phase transitions.} Z. Phys. \textbf{253} (1972), 147–-161.
\bibitem{complexpoly} T. Sheil-Small, \textit{Complex Polynomials.} Cambridge University Press, 2002.
\bibitem{slater} L. J. Slater, \textit{Confluent Hypergeometric Functions.} Cambridge University Press, 1960.
\bibitem{szego} G. Szeg\"o, \textit{Orthogonal Polynomials.} American Mathematical Society, 1939.
\bibitem{tavare} S. Tavar\'e, \textit{Lines of descent and genealogical processes, and their application in population genetics models.} Theoret. Popn. Biol. \textbf{26} (1984), 119--164.
\bibitem{avr} A. Vandenberg-Rodes, \textit{A limit theorem for particle current in the symmetric exclusion process.} Elect. Comm. Prob. \textbf{15} (2010), 240--253.

\end{thebibliography}
\end{document}